\documentclass{article}
\usepackage[utf8]{inputenc}
\usepackage[T1]{fontenc}
\usepackage{lmodern}
\usepackage{graphicx}
\usepackage{multicol}
\usepackage[top=2.5cm, bottom=2.5cm, left=2.5cm , right=2.5cm]{geometry}

\usepackage{pgf,tikz,pgfplots}
\pgfplotsset{compat=1.14}
\usepackage{mathrsfs}
\usepackage{hyperref}
\usetikzlibrary{arrows}
\usepackage{amsmath,amssymb,amsthm}
\usepackage{amsthm}

\newtheorem{theorem}{Theorem}[section]
\newtheorem{proposition}[theorem]{Proposition}
\newtheorem{lemma}[theorem]{Lemma}

\theoremstyle{definition}
\newtheorem{definition}[theorem]{Definition}

\theoremstyle{remark}
\newtheorem{remark}[theorem]{Remark}

\newcounter{exocmpt}[section]
\newcounter{partcmpt}[exocmpt]
\newcounter{questioncmpt}[exocmpt]
\newcounter{subquestioncmpt}[questioncmpt]

\theoremstyle{definition}

\newcommand{\vsp}{\vspace{0.2cm}}

\newcommand{\cp}{\mathbb{C}\mathbb{P}}
\newcommand{\h}{\mathbb{H}}
\newcommand{\s}{\mathbb{S}}

\newcommand{\rp}{\mathbb{R}\mathbb{P}}

\newcommand{\real}{\mathbb{R}}

\newcommand{\integer}{\mathbb{Z}}
\newcommand{\pinteger}{\mathbb{N}}

\newcommand{\etoile}{^{\ast}}

\begin{document}

\title{Examples of reflective projective billiards and outer ghost billiards}
\author{Corentin Fierobe}
\date\today
\maketitle


\newcommand{\bref}{\mathbb{P}(T\cp^2)}
\newcommand{\az}{\mbox{az}}
\newcommand{\negl}{\mbox{o}}
\newcommand{\projspace}{\mathcal{P}}
\newcommand{\polbill}{\text{Pol}}
\newcommand{\cppol}{\text{CPol}}
\newcommand{\huit}{\frac{1}{8}}

\begin{abstract}
In the class of projective billiards, which contains the usual billiards, we exhibit counter-examples to Ivrii's conjecture, which states that in any planar billiard with smooth boundary the set of periodic orbits has zero measure. The counter-examples are polygons admitting a $2$-parameters family of $n$-periodic orbits, with $n$ being either $3$ or any even integer grower than $4$.
\end{abstract}

\tableofcontents

\section{Introduction}

In the theory of billiard, describing properties of the periodic orbits of the billiard inside a domain is the subject of many studies. A famous conjecture, due to Ivrii states the following: \textit{in any smooth euclidean planar billiard, the set of periodic orbits has zero measure}. The proof of the conjecture was made in the case of a billiard with a regular analytic convex boundary, see \cite{vasiliev}. The conjecture is also true for $3$-periodic orbits, and was proved in \cite{bary, rychlik, stojanov, vorobets, wojtkowski}. For $4$-periodic orbits, it was proven in \cite{glutkud1, glutkud2} and a complete classification of $4$-reflective (defined later) complex analytic billiards was presented in \cite{glut}. Ivrii's conjecture was also studied in manifolds of constant curvature: it was proven to be true for $k=3$ in the hyperbolic plane $\h^2$, see \cite{VKNZ}; the case of the sphere $\s^2$ was apparently firstly studied in \cite{bary_introuvable}, as quoted in \cite{VKNZ} but we were not able to find the correponding paper. The sphere is in fact an example of space were Ivrii's conjecture is not true, and we can find a classification of all counter-examples in \cite{VKNZ}.

\begin{figure}[!ht]
\centering
\input{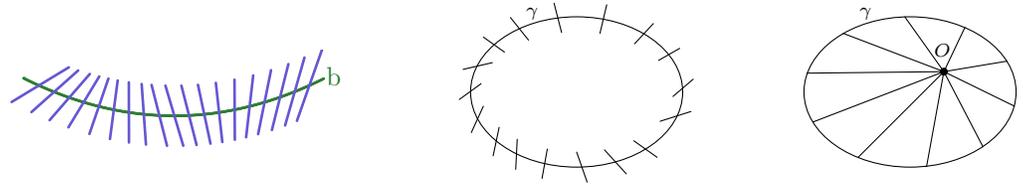}
\hspace{1cm}
\begin{tikzpicture}[line cap=round,line join=round,>=triangle 45,x=1.0cm,y=1.0cm]
\clip(-1.6,-1.2) rectangle (1.6,1.2);
\draw [rotate around={0:(0,0)}] (0,0) ellipse (1.41cm and 1cm);
\draw (-1.56,-0.1)-- (-1.27,0.11);
\draw (-1.51,0.3)-- (-1.12,0.39);
\draw (-1.24,0.73)-- (-0.96,0.53);
\draw (-0.88,0.99)-- (-0.64,0.69);
\draw (-0.3,1.18)-- (-0.2,0.79);
\draw (0.4,1.16)-- (0.31,0.78);
\draw (1.02,0.91)-- (0.77,0.64);
\draw (1.37,0.58)-- (1.09,0.42);
\draw (1.55,0.17)-- (1.28,-0.06);
\draw (1.07,-0.87)-- (0.76,-0.65);
\draw (0.66,-1.09)-- (0.37,-0.77);
\draw (0.14,-1.2)-- (0,-0.8);
\draw (-0.46,-1.15)-- (-0.39,-0.76);
\draw (-0.81,-1.03)-- (-0.79,-0.63);
\draw (-1.11,-0.85)-- (-1.03,-0.47);
\draw (-1.4,-0.54)-- (-1.24,-0.19);
\draw (1.11,-0.39)-- (1.52,-0.26);
\begin{scriptsize}
\draw[color=black] (-0.59,1.04) node {$\gamma$};
\end{scriptsize}
\end{tikzpicture}
\hspace{1cm}
\begin{tikzpicture}[line cap=round,line join=round,>=triangle 45,x=1.0cm,y=1.0cm]
\clip(-1.6,-1.2) rectangle (1.6,1.2);
\draw [rotate around={0:(0,0)}] (0,0) ellipse (1.41cm and 1cm);
\draw (0.45,0.27)-- (-0.07,1);
\draw (0.45,0.27)-- (-0.89,0.78);
\draw (0.45,0.27)-- (-1.37,0.25);
\draw (0.45,0.27)-- (-1.29,-0.4);
\draw (0.45,0.27)-- (-0.69,-0.87);
\draw (0.45,0.27)-- (0.22,-0.99);
\draw (0.45,0.27)-- (0.97,-0.73);
\draw (0.45,0.27)-- (1.39,-0.19);
\draw (0.45,0.27)-- (1.29,0.41);
\draw (0.45,0.27)-- (0.73,0.86);
\begin{scriptsize}
\draw[color=black] (-0.59,1.04) node {$\gamma$};
\fill [color=black] (0.45,0.27) circle (1.5pt);
\draw[color=black] (0.43,0.55) node {$O$};
\end{scriptsize}
\end{tikzpicture}
\caption{Left: a curve $b$ endowed with a field of transverse line. Center: a convex closed curve $\gamma$ with a field of transverse lines. Right: the same curve $\gamma$ with a so-called \textit{centrally-projective} field of transverse lines.}
\label{fig1}
\end{figure}

In this paper we study a generalization of usual billiards, as it is described in  \cite{taba_projectif}. We consider the so called \textit{projective billiards}, which are billiards having their boundaries endowed with a field of transverse lines, called \textit{projective field of lines} (see Figure \ref{fig1}), and defining at each point a new reflection law 
which we call \textit{projective reflection law}: suppose that at a point $p$ of a curve $\gamma$, the transverse line at $p$ is $L$. A line $\ell$ hitting $\gamma$ at $p$ will be reflected into a line $\ell'$ if the lines $(\ell,\ell',L,T_p\gamma)$ are harmonic, meaning that their cross-ratio is $-1$. We recall that the cross-ratio of four lines through the same point $p$ is defined as the cross-ratio of their intersection with a fifth line not passing through $p$, and this doesn't depend on the line taken. Note that this definition makes the projective plane $\rp^2$ a convenient space to study projective billiards.

We can therefore study the billiard dynamics on projective billiard chosing this new law to determine the way a particle bouces inside a table. Note that, given any closed curve $\gamma$, the usual billiard dynamic inside $\gamma$ is also a projective one by choosing the transverse lines to be the normal lines to $\gamma$. In this paper we study examples of projective billiards inside polygons. They are defined as follows:

\vspace{0.2cm}
EXAMPLE 1. Consider a smooth convex planar curve $\gamma$ and an other point $O$ inside $\gamma$ called the origin. Define a transverse field of lines on $\gamma$ by taking the lines passing through the origin, see Figure \ref{fig:centrally_projective}. As in \cite{taba_projectif}, this billiard is called \textit{centrally-projective} and has many interesting properties, including the fact that the billiard transformation has an invariant area form (\cite{taba_projectif}, Theorem D).

In the case when $\gamma$ is a polygon $\mathcal{P}$ with $n$-vertices, $P_0,\ldots,P_{n-1}$, we call this billiard a \textit{centrally-projective polygon}, and denote it by $\cppol(O;P_0,\ldots,P_{n-1})$. When the poygon $\mathcal{P}$ is regular and the origin is the intersection of its great diagonals, we call it \textit{centrally-projective regular polygon}.

\begin{figure}[!ht]
\centering
\begin{tikzpicture}[line cap=round,line join=round,>=triangle 45,x=1.5cm,y=1.5cm]
\clip(-1.6,-1.2) rectangle (1.6,1.2);
\draw [rotate around={0:(0,0)}] (0,0) ellipse (2.12cm and 1.5cm);
\draw (0.45,0.27)-- (-0.07,1);
\draw (0.45,0.27)-- (-0.89,0.78);
\draw (0.45,0.27)-- (-1.37,0.25);
\draw (0.45,0.27)-- (-1.29,-0.4);
\draw (0.45,0.27)-- (-0.69,-0.87);
\draw (0.45,0.27)-- (0.22,-0.99);
\draw (0.45,0.27)-- (0.97,-0.73);
\draw (0.45,0.27)-- (1.39,-0.19);
\draw (0.45,0.27)-- (1.29,0.41);
\draw (0.45,0.27)-- (0.73,0.86);
\begin{scriptsize}
\draw[color=black] (-0.59,1.04) node {$\gamma$};
\fill [color=black] (0.45,0.27) circle (1.5pt);
\draw[color=black] (0.43,0.55) node {$O$};
\end{scriptsize}
\end{tikzpicture}
\hspace{2cm}
\begin{tikzpicture}[line cap=round,line join=round,>=triangle 45,x=0.7cm,y=0.7cm]
\clip(-3.7,-2.5) rectangle (3.7,3.5);
\draw (-2,-2)-- (-3,1);
\draw (-3,1)-- (2,3);
\draw (2,3)-- (3,-2);
\draw (-2,-2)-- (3,-2);
\draw [dotted] (-2,-2)-- (2,3);
\draw [dotted] (-3,1)-- (3,-2);
\draw [dash pattern=on 1pt off 1pt] (-0.57,-0.21)-- (-2.94,-0.84);
\draw [dash pattern=on 1pt off 1pt] (-0.57,-0.21)-- (-0.48,2.46);
\draw [dash pattern=on 1pt off 1pt] (-0.57,-0.21)-- (2.95,0.71);
\draw [dash pattern=on 1pt off 1pt] (-0.57,-0.21)-- (0.4,-2.46);
\draw (-3.39,0.53)-- (-2.2,0.21);
\draw (-3.17,-0.15)-- (-2.06,-0.18);
\draw (-2.94,-0.84)-- (-1.93,-0.58);
\draw (-2.68,-1.62)-- (-1.78,-1.02);
\draw (-1.5,-2.46)-- (-1.11,-1.51);
\draw (-0.57,-1.51)-- (-0.56,-2.46);
\draw (-0.01,-1.51)-- (0.4,-2.46);
\draw (1.25,-2.46)-- (0.48,-1.51);
\draw (1.06,-1.51)-- (2.25,-2.46);
\draw (3.33,-1.23)-- (2.3,-0.96);
\draw (3.2,-0.55)-- (2.2,-0.46);
\draw (3.07,0.1)-- (2.1,0.01);
\draw (2.01,0.46)-- (2.95,0.71);
\draw (2.82,1.34)-- (1.92,0.93);
\draw (1.82,1.44)-- (2.68,2.04);
\draw (2.55,2.68)-- (1.72,1.92);
\draw (1.36,3.19)-- (0.7,2.03);
\draw (0.53,2.86)-- (0.15,1.81);
\draw (-0.48,2.46)-- (-0.51,1.54);
\draw (-0.96,1.37)-- (-1.16,2.19);
\draw (-1.96,1.87)-- (-1.48,1.16);
\draw (-1.95,0.97)-- (-2.65,1.58);
\begin{scriptsize}
\fill [color=black] (-2,-2) circle (1.5pt);
\draw[color=black] (-2.1,-2.13) node {$P_0$};
\fill [color=black] (-3,1) circle (1.5pt);
\draw[color=black] (-3.2,1.09) node {$P_1$};
\fill [color=black] (2,3) circle (1.5pt);
\draw[color=black] (2.12,3.2) node {$P_2$};
\fill [color=black] (3,-2) circle (1.5pt);
\draw[color=black] (3.19,-2.05) node {$P_3$};
\fill [color=black] (-0.57,-0.21) circle (1.5pt);
\draw[color=black] (-0.7,0.04) node {$O$};
\draw[color=black] (-1.69,-0.34) node {$L_0$};
\draw[color=black] (-0.33,1.17) node {$L_1$};
\draw[color=black] (1.88,0.26) node {$L_2$};
\draw[color=black] (-0.24,-1.35) node {$L_3$};
\end{scriptsize}
\end{tikzpicture}
\caption{Left: a convex closed curve $\gamma$ with a so-called \textit{centrally-projective} field of transverse lines through an origin $O$. Right: a centrally-projective quadrilateral $P_0P_1P_2P_3$ with origin $O$.}
\label{fig:centrally_projective}
\end{figure}
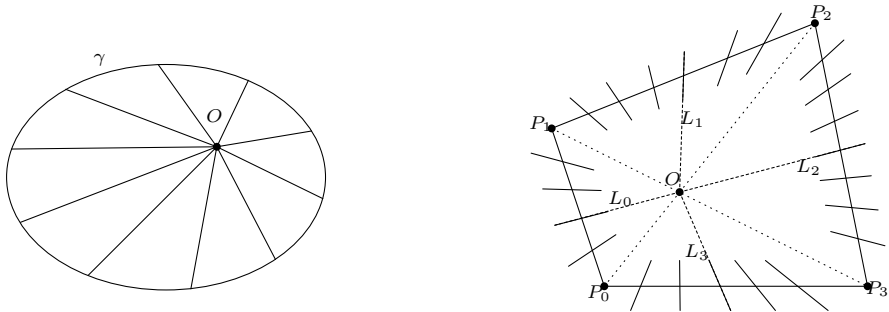

\vspace{0.2cm}
EXAMPLE 2. Consider a triangle $P_0P_1P_2$ in the (projective) plane. On a side $P_iP_{i+1}$, define a field of transverse lines by taking the lines passing through the remaining vertex of the triangle, $P_{i-1}$, see Figure \ref{fig:spherical_presentation}. We call this projective billiard \textit{right-spherical billiard based at $(P_0,P_1,P_2)$.} and denote it by $\mathcal{S}(P_0,P_1,P_2)$.

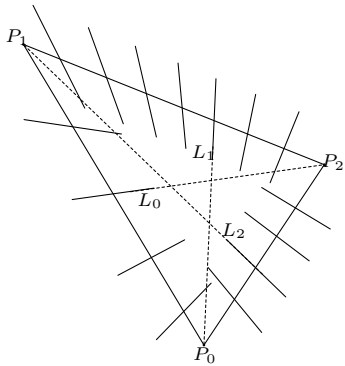
\begin{figure}[!ht]
\centering
\begin{tikzpicture}[line cap=round,line join=round,>=triangle 45,x=0.8cm,y=0.8cm]
\clip(-3.4,-2.5) rectangle (2.4,4);
\draw (2,1)-- (0,-2);
\draw (0,-2)-- (-3,3);
\draw (-3,3)-- (2,1);
\draw (-2.84,3.65)-- (-1.98,1.94);
\draw (-1.92,3.28)-- (-1.34,1.68);
\draw (-1.14,2.97)-- (-0.79,1.46);
\draw (-0.43,2.69)-- (-0.3,1.27);
\draw (0.2,2.43)-- (0.14,1.09);
\draw (0.86,2.17)-- (0.6,0.9);
\draw (1.58,1.88)-- (1.1,0.7);
\draw (-2.99,1.75)-- (-1.37,1.51);
\draw (-2.19,0.42)-- (-0.83,0.61);
\draw (-1.43,-0.84)-- (-0.32,-0.25);
\draw (-0.79,-1.91)-- (0.12,-0.97);
\draw (0.07,-0.71)-- (0.96,-1.79);
\draw (1.36,-1.2)-- (0.38,-0.25);
\draw (0.68,0.21)-- (1.75,-0.6);
\draw (0.96,0.62)-- (2.1,-0.07);
\draw [line width=0.4pt,dash pattern=on 1pt off 1pt] (0.97,-0.82)-- (-3,3);
\draw [line width=0.4pt,dash pattern=on 1pt off 1pt] (-1.22,0.55)-- (2,1);
\draw [line width=0.4pt,dash pattern=on 1pt off 1pt] (0.15,1.3)-- (0,-2);
\begin{scriptsize}
\fill [color=black] (2,1) circle (0.5pt);
\draw[color=black] (2.15,1.04) node {$P_2$};
\fill [color=black] (0,-2) circle (0.5pt);
\draw[color=black] (0.02,-2.17) node {$P_0$};
\fill [color=black] (-3,3) circle (0.5pt);
\draw[color=black] (-3.1,3.12) node {$P_1$};
\draw[color=black] (0.5,-0.1) node {$L_2$};
\draw[color=black] (-0.9,0.4) node {$L_0$};
\draw[color=black] (-0.0,1.2) node {$L_1$};
\end{scriptsize}
\end{tikzpicture}
\caption{The right-spherical billiard $\mathcal{S}(P_0,P_1,P_2)$ with each one of its fields of transverse lines}
\label{fig:spherical_presentation}
\end{figure}

Its name comes from a famous construction on the two-dimensionnal sphere $\s^2$: consider a triangle $Q_0Q_1Q_2$ on $\s^2$ having right angles at all its vertices (its area is one eighth of the total surface of $\s^2$). This triangle is a well-known example of $2$-reflective billiard on the sphere, as explained in \cite{bary_introuvable,VKNZ}. Now consider the projection of $\s^2$ from its center to a plane $\mathcal{P}$ tangent to the sphere at a point located inside $Q_0Q_1Q_2$: $Q_0Q_1Q_2$ is projected on a triangle $P_0P_1P_2$ of $\mathcal{P}$, and the normal geodesics to the sides of $Q_0Q_1Q_2$ are projected into the projective field of lines making $P_0P_1P_2$ a right-spherical billiard.

\vspace{0.2cm}
Here we consider a simplified version of the billiard dynamic, where the name ghost come from, and which is enough to study examples of curves having the reflectivity property (defined later). In the particular case of a polygon $\mathcal{P}=P_0\cdots P_{n-1}$ endowed with a field of transverse lines on each of its sides, we extend the definition of edges of $\mathcal{P}$ to be the whole lines supporting the usual edges. We force the successive bounces of a billiard trajectory to be on successive edges of $\mathcal{P}$, by forgetting any possible obstacles on the ball trajectory between two consecutive edges. More precisely, fix two distinct points $M_0\in P_0P_1$ and $M_1\in P_1P_2$ on two consecutive edges. We define  the orbit of $(M_0,M_1)$ as follows:

\begin{definition}
The \textit{orbit of $(M_0,M_1)$} is the sequence $(M_k)_{k\in\integer}$ where for each $k$, $M_{k}$ is a point on $P_kP_{k+1}$ (where $k$ is taken modulo $n$) and the lines $M_{k-1}M_k$ and $M_kM_{k+1}$ are symmetric with respect to the projective reflection law at $M_k$. 

Let $m=kn$ be a positive mutliplier of $n$. The orbit is called\textit{ $m$-periodic} if $M_{m}M_{m+1}=M_0M_1$. The corresponding projective billiard is said to be \textit{$m$-reflective} if there is a non-empty open subset $U\times V\subset P_0P_1\times P_1P_2$ of $(M_0,M_1)$ whose orbits are $m$-periodic.
\end{definition}

In the following, we use these examples to construct reflective projective billiards (having the additional particular property that $U\times V=P_0P_1\times P_1P_2$, which can be seen as a consequence of analyticity of the billiard map). First, since the right-spherical billiard is obtained by projecting a certain $3$-reflective billiard of $\s^2$ on the plane, we expect that it is also $2$-reflective (see Figure \ref{fig:3_reflective}). Indeed, we have the

\begin{proposition}
\label{res:spherical_3_reflective}
The right-spherical billiard, $\mathcal{S}(P_0,P_1,P_2)$, is $3$-reflective.
\end{proposition}

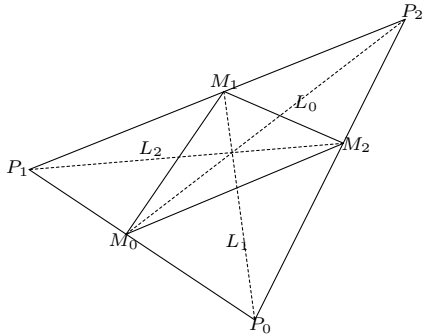
\begin{figure}[!h]
\centering
\begin{tikzpicture}[line cap=round,line join=round,>=triangle 45,x=1cm,y=1cm]
\clip(-3.5,-2.5) rectangle (3,2.5);
\draw (-3,0)-- (0,-2);
\draw (0,-2)-- (2,2);
\draw (2,2)-- (-3,0);
\draw (-1.71,-0.86)-- (-0.41,1.04);
\draw (-0.41,1.04)-- (1.18,0.35);
\draw [line width=0.4pt,dash pattern=on 1pt off 1pt] (-1.71,-0.86)-- (2,2);
\draw [line width=0.4pt,dash pattern=on 1pt off 1pt] (-0.41,1.04)-- (0,-2);
\draw [line width=0.4pt,dash pattern=on 1pt off 1pt] (1.18,0.35)-- (-3,0);
\draw (-1.71,-0.86)-- (1.18,0.35);
\begin{scriptsize}
\draw[color=black] (-3.15,0.02) node {$P_1$};
\draw[color=black] (0.08,-2.07) node {$P_0$};
\draw[color=black] (2.1,2.1) node {$P_2$};
\draw[color=black] (-1.74,-0.95) node {$M_0$};
\draw[color=black] (-0.4,1.17) node {$M_1$};
\draw[color=black] (1.35,0.33) node {$M_2$};
\draw[color=black] (0.68,0.88) node {$L_0$};
\draw[color=black] (-0.23,-0.97) node {$L_1$};
\draw[color=black] (-1.38,0.28) node {$L_2$};
\end{scriptsize}
\end{tikzpicture}
\caption{The spherical billiard $\mathcal{S}(P_0,P_1,P_2)$ and a triangular orbit $(M_0,M_1,M_2)$ obtained by reflecting any segment $M_0M_1$ two times}
\label{fig:3_reflective}
\end{figure}

In Section \ref{sec:n_4}, we construct classes of $4$-reflective projective billiards which cannot be deduced from one another. More precisely, we build the first one by \textit{gluing} two right-spherical billiards. The second one is a centrally-projective polygon, as stated by

\begin{proposition}
\label{res:quadrilateral_4_reflective}
Let $P_0,P_1,P_2,P_3$ be a non-degenerate quadrilateral and $O$ the intersection point of its diagonals. The centrally-projective polygon $\cppol(O;P_0,P_1,P_2,P_3)$ is $4$-reflective.
\end{proposition}

In Section \ref{sec:n_even}, we generalize the case $n=4$ to any even integer $n\geq 4$, with restrictions on the polygon however:

\begin{proposition}
\label{res:regular_centrally_reflective}
Let $n\geq 4$ be even. Let $P_0\cdots P_{n-1}$ be a regular polygon and $O$ be the intersection point of its diagonals. The centrally-projective regular billiard $\cppol(O;P_0,\ldots,P_{n-1})$ is $n$-reflective.
\end{proposition}

This result is not true for $n$ even. However, if one allows two times more bounces, it becomes true, and not only for regular polygons. More precisely, we state the following

\begin{proposition}
\label{res:centrally_2n_reflective}
Let $n\geq 3$ be odd. Let $P_0\cdots P_{n-1}$ be a polygon and $O$ be any point not lying on the edges of $P$. The centrally-projective billiard $\cppol(O;P_0,\ldots,P_{n-1})$ is $2n$-reflective.
\end{proposition}

The results are proved in the following order: we prove Proposition \ref{res:spherical_3_reflective} in Section \ref{sec:n_3}, Proposition \ref{res:quadrilateral_4_reflective} in Section \ref{sec:n_4} (and present another example of $4$-reflective billiard derived from the right-spherical model), Proposition \ref{res:regular_centrally_reflective} in Section \ref{sec:n_even}, and Proposition \ref{res:centrally_2n_reflective} in Section \ref{sec:centrally_proj_poly}.

Knowing these examples, many questions arise: are there polygonal examples of $n$-reflective projective billiards with $n$ odd? Can we find centrally-projective billiards which are $n$-reflective but not built on polygons? Another more difficult question would be: what are the $n$-reflective projective billiards (with the requirement that the boundary has a certain class of smoothness)? Answering this question would be a great step in solving Ivrii's conjecture. In \cite{fierobe_projective_3reflective}, we show that the only $3$-reflective projective billiard table with analytic boundary is the right-spherical billiard, as described in example described in Section \ref{sec:n_3}. 

We would like to conclude this introduction by saying that you can find, on the webpage of the author (which was \cite{fierobe_simulation_projective_billiards} at the time of writing), a simulation of projective billiards in polygons in which you can visualize their dynamics.

\section{$3$-reflectivity of the right-spherical billiard}
\label{sec:n_3}

In this section, we prove Proposition \ref{res:spherical_3_reflective}: the right-spherical billiard is $3$-reflective. We first introduce some notations:

Let $P_0$, $P_1$, $P_2$ be three points not on the same line. For $i=0,1,2$, let $\gamma_i$ be the line $P_{i}P_{i+1}$. For any $M\in\gamma_i$, let $L_i(M)$ be the line $MP_{i+2}$ ($i$ is seen modulo $3$), that is $L_i(M)$ is line joining $M$ and the only point $P_j$ which do lie on $\gamma_i$. The projective billiard table defined by the $\gamma_i$ and the $L_i$ is the right-spherical billiard $\mathcal{S}(P_0,P_1,P_2)$ (see Figure \ref{fig:spherical_presentation}).

\begin{proposition}
\label{prop:3_reflective}
Any $(M_0,M_1)\in\gamma_0\times\gamma_1$ with $M_0\neq M_1$ determines a $3$-periodic orbit inside the right-spherical billiard $\mathcal{S}(P_0,P_1,P_2)$.
\end{proposition}

\begin{proof}
This proof was found by Simon Allais in a talk we had about harmonicity in a projective space. Let $M_2\in\gamma_2$ be such that $M_0M_1$, $M_1M_2$, $\gamma_1$, $L_1(M_1)$ are harmonic lines. Define $M_2'\in\gamma_2$ similarly: $M_0M_1$, $M_0M_2'$, $\gamma_0$, $L_0(M_0)$ are harmonic lines.

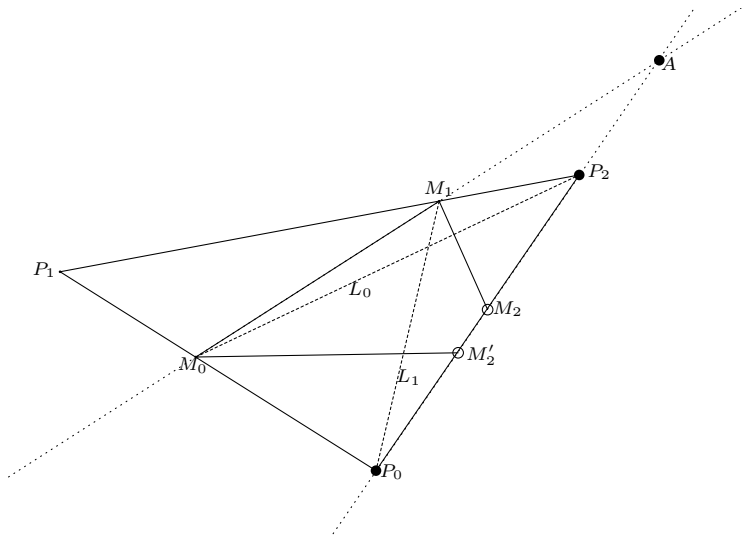
\begin{figure}[!ht]
\centering
\begin{tikzpicture}[line cap=round,line join=round,>=triangle 45,x=1.4cm,y=1.4cm]
\clip(-3.5,-2.5) rectangle (3.5,2.5);
\draw (-3,0)-- (0,-1.89);
\draw (0,-1.89)-- (1.93,0.92);
\draw (1.93,0.92)-- (-3,0);
\draw (-1.71,-0.81)-- (0.6,0.67);
\draw (0.6,0.67)-- (1.06,-0.36);
\draw [line width=0.4pt,dash pattern=on 1pt off 1pt] (-1.71,-0.81)-- (1.93,0.92);
\draw [line width=0.4pt,dash pattern=on 1pt off 1pt] (0.6,0.67)-- (0,-1.89);
\draw [dotted,domain=-4.91:4.59] plot(\x,{(--0.66--1.48*\x)/2.31});
\draw [dotted,domain=-4.91:4.59] plot(\x,{(--3.66-2.81*\x)/-1.93});
\draw (-1.71,-0.81)-- (0.78,-0.77);
\begin{scriptsize}
\fill [color=black] (-3,0) circle (0.5pt);
\draw[color=black] (-3.15,0.02) node {$P_1$};
\fill [color=black] (0,-1.89) circle (2pt);
\draw[color=black] (0.15,-1.9) node {$P_0$};
\fill [color=black] (1.93,0.92) circle (2pt);
\draw[color=black] (2.12,0.95) node {$P_2$};
\fill [color=black] (-1.71,-0.81) circle (0.5pt);
\draw[color=black] (-1.74,-0.89) node {$M_0$};
\fill [color=black] (0.6,0.67) circle (0.5pt);
\draw[color=black] (0.6,0.79) node {$M_1$};
\draw [color=black] (1.06,-0.36) circle (2pt);
\draw[color=black] (1.25,-0.35) node {$M_2$};
\draw[color=black] (-0.15,-0.17) node {$L_0$};
\draw[color=black] (0.31,-1) node {$L_1$};
\fill [color=black] (2.69,2.01) circle (2pt);
\draw[color=black] (2.78,1.98) node {$A$};
\draw [color=black] (0.78,-0.77) circle (2pt);
\draw[color=black] (1,-0.79) node {$M_2'$};
\end{scriptsize}
\end{tikzpicture}
\caption{As in the proof of Proposition \ref{prop:3_reflective}, both quadruples of points $(M_2,A,P_0,P_2)$ and $(M_2',A,P_0,P_2)$ are harmonic, hence necessarily $M_2=M_2'$.}
\label{fig:3_reflective_proof1}
\end{figure}

Let us first show that necessarily $M_2=M_2'$ (see Figure \ref{fig:3_reflective_proof1}). Consider the line $\gamma_2$ and let $A$ be its point of intersection with $M_0M_1$. Let us consider harmonic quadruples of points on $\gamma_2$. By harmonicity of the previous defined lines passing through $M_1$, the quadruple of points $(A,M_2,P_2,P_0)$ is harmonic. Doing the same with the lines passing through $M_0$, the quadruple of points $(A,M_2',P_2,P_0)$ is harmonic. Hence $M_2=M_2'$ since the projective transformation defining the cross-ratio is one to one.

Now let us prove that the lines $M_1M_2$, $M_0M_2$, $\gamma_2$, $L_2(M_2)$ are harmonic lines. Consider the line $\gamma_0$: $M_1M_2$ intersects it at a certain point denoted by $B$, $M_2M_0$ at $M_0$, $\gamma_2$ at $P_0$ and $L_2(M_2)$ at $P_1$. But the quadruple of points $(B,M_0,P_0,P_1)$ is harmonic since there is a reflection law at $M_1$ whose lines intersect $\gamma_0$ exactly in those points.
\end{proof}



\section{Two $4$-reflective projective billiards}
\label{sec:n_4}

In this section we construct two classes of $4$-reflective projective billiards: the first one is obtained by gluing together two right-spherical billiards, the second one is obtained from a centrally-projective quadrilateral whose origin is the intersection point of its diagonals.

\subsection{The converging mirrors}

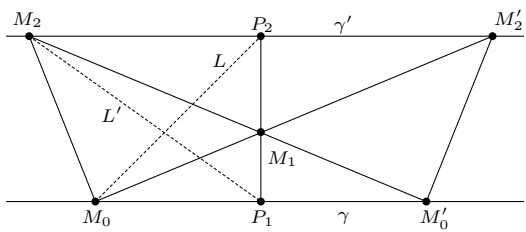
\begin{figure}[h!]
\centering
\begin{tikzpicture}[line cap=round,line join=round,>=triangle 45,x=2.2cm,y=2.2cm]
\clip(-1.54,-0.32) rectangle (1.59,1.41);
\draw [domain=-1.54:1.59] plot(\x,{(--1-0*\x)/1});
\draw [domain=-1.54:1.59] plot(\x,{(-0-0*\x)/1.5});
\draw (0,1)-- (0,0);
\draw (-1,0)-- (-1.4,1);
\draw (-1.4,1)-- (1,0);
\draw (1,0)-- (1.4,1);
\draw (1.4,1)-- (-1,0);
\draw [dash pattern=on 1pt off 1pt] (-1.4,1)-- (0,0);
\draw [dash pattern=on 1pt off 1pt] (-1,0)-- (0,1);
\begin{scriptsize}
\fill [color=black] (0,1) circle (1.5pt);
\draw[color=black] (0.5,1.07) node {$\gamma'$};
\draw[color=black] (0.5,-0.1) node {$\gamma$};
\draw[color=black] (0.01,1.07) node {$P_2$};
\fill [color=black] (0,0) circle (1.5pt);
\draw[color=black] (0.01,-0.1) node {$P_1$};
\fill [color=black] (-1,0) circle (1.5pt);
\draw[color=black] (-0.99,-0.1) node {$M_0$};
\fill [color=black] (-1.4,1) circle (1.5pt);
\draw[color=black] (-1.41,1.1) node {$M_2$};
\fill [color=black] (1,0) circle (1.5pt);
\draw[color=black] (1.05,-0.1) node {$M_0'$};
\fill [color=black] (1.4,1) circle (1.5pt);
\draw[color=black] (1.5,1.1) node {$M_2'$};
\fill [color=black] (0,0.42) circle (1.5pt);
\draw[color=black] (0.13,0.26) node {$M_1$};
\draw[color=black] (-0.9,0.53) node {$L'$};
\draw[color=black] (-0.25,0.85) node {$L$};
\end{scriptsize}
\end{tikzpicture}
\caption{billiard deduced from the right-spherical billiard $\mathcal{S}(P_0,P_1,P_2)$, where the point $P_0$ is at infinity, hence the lines $P_0P_1$, $P_0P_2$ and all tranverse lines $L_1(M)$ on $P_1P_2$ are parallel, making the reflection law on $P_1P_2$ as the usual one.}
\label{fig:4_deduced}
\end{figure}

In $\real^2$, consider two distinct parallel lines $\gamma$ and $\gamma'$. Choose a normal line intersecting $\gamma$ at a point $P_1$ and $\gamma'$ at a point $P_2$. On $\gamma$, define a field of transverse lines by taking the lines $L(M)$ going from a point $M\in\gamma$ to $P_2$. Do the same on $\gamma'$ with $P_1$ to form the field of lines $L'$. Consider the projective billiard dynamics made by a particle bouncing alternatively $\gamma$ and $\gamma'$ with their respective fields of transverse lines. Call this construction the \textit{converging mirrors}. We claim that

\begin{proposition}
The converging mirrors are $4$-reflective.
\end{proposition}

\begin{proof}
Complete $\real^2$ with a line at infinity $\ell$ to form the projective plane. Consider the point $P_0\in\ell$ where $\gamma$ and $\gamma'$ intersect. Choose any orbit $(M_0,M_1,M_2)$ of the spherical billiard $\mathcal{S}(P_0,P_1,P_2)$, where $M_0\in \gamma$, $M_1\in P_1P_2$ and $M_2\in\gamma'$ (see Figure \ref{fig:4_deduced}). Now consider the axial reflection with respect to the line $P_1P_2$. It leaves $P_1P_2$, $\gamma$, $\gamma'$ and $P_0$ invariant, hence the orbit $(M_0,M_1,M_2)$ is transformed into another orbit $(M_0',M_1,M_2')$ of $\mathcal{S}(P_0,P_1,P_2)$, where $M_0'\in \gamma$ and $M_2'\in\gamma'$.

But since the transverse lines on the side $P_1P_2$ of $\mathcal{S}(P_0,P_1,P_2)$ are orthogonal to the line $P_1P_2$ (they pass through $P_0$, as well as $\gamma$ and $\gamma'$), the reflection law on $P_1P_2$ is the usual one: the lines $M_0M_1$ and $M_1M_2$ make the same angle with $P_1P_2$. Hence $M_0,M_1,M_2'$ are on the same line, and so are $M_2,M_1,M_0'$. Therefore, $M_0,M_2',M_0',M_2$ is a $4$-periodic orbit. 
\end{proof}
%
%

\subsection{The centrally-projective quadrilateral}

In this subsection, we prove Proposition \ref{res:quadrilateral_4_reflective}. We first introduce some notations:

Let $P_0$, $P_1$, $P_2$, $P_3$ be four points, no three of them being on the same line. For $i=0\ldots 3$ (seen modulo $4$), let $\gamma_i$ be the line $P_{i}P_{i+1}$. Write $O$ to be the point of intersection of the lines $P_0P_2$ and $P_1P_3$ (diagonals). For any $M\in\gamma_i$, let $L_i(M)$ be the line $OM$. The projective billiard table defined by the $\gamma_i$ and the $L_i$ is what we call the \textit{centrally-projective quadrilateral} $\cppol(O;P_0,P_1,P_2,P_3)$ (see Figure \ref{fig:quadrilateral_billiard}).

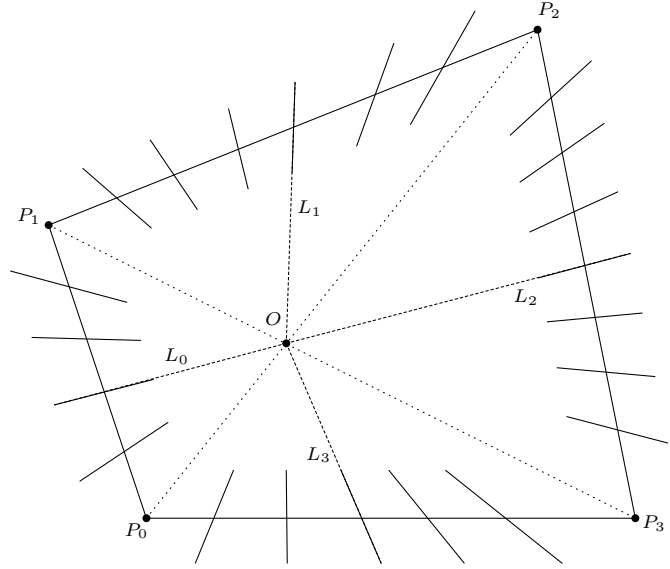
\begin{figure}[!h]
\centering
\begin{tikzpicture}[line cap=round,line join=round,>=triangle 45,x=1.3cm,y=1.3cm]
\clip(-3.7,-2.5) rectangle (3.7,3.5);
\draw (-2,-2)-- (-3,1);
\draw (-3,1)-- (2,3);
\draw (2,3)-- (3,-2);
\draw (-2,-2)-- (3,-2);
\draw [dotted] (-2,-2)-- (2,3);
\draw [dotted] (-3,1)-- (3,-2);
\draw [dash pattern=on 1pt off 1pt] (-0.57,-0.21)-- (-2.94,-0.84);
\draw [dash pattern=on 1pt off 1pt] (-0.57,-0.21)-- (-0.48,2.46);
\draw [dash pattern=on 1pt off 1pt] (-0.57,-0.21)-- (2.95,0.71);
\draw [dash pattern=on 1pt off 1pt] (-0.57,-0.21)-- (0.4,-2.46);
\draw (-3.39,0.53)-- (-2.2,0.21);
\draw (-3.17,-0.15)-- (-2.06,-0.18);
\draw (-2.94,-0.84)-- (-1.93,-0.58);
\draw (-2.68,-1.62)-- (-1.78,-1.02);
\draw (-1.5,-2.46)-- (-1.11,-1.51);
\draw (-0.57,-1.51)-- (-0.56,-2.46);
\draw (-0.01,-1.51)-- (0.4,-2.46);
\draw (1.25,-2.46)-- (0.48,-1.51);
\draw (1.06,-1.51)-- (2.25,-2.46);
\draw (3.33,-1.23)-- (2.3,-0.96);
\draw (3.2,-0.55)-- (2.2,-0.46);
\draw (3.07,0.1)-- (2.1,0.01);
\draw (2.01,0.46)-- (2.95,0.71);
\draw (2.82,1.34)-- (1.92,0.93);
\draw (1.82,1.44)-- (2.68,2.04);
\draw (2.55,2.68)-- (1.72,1.92);
\draw (1.36,3.19)-- (0.7,2.03);
\draw (0.53,2.86)-- (0.15,1.81);
\draw (-0.48,2.46)-- (-0.51,1.54);
\draw (-0.96,1.37)-- (-1.16,2.19);
\draw (-1.96,1.87)-- (-1.48,1.16);
\draw (-1.95,0.97)-- (-2.65,1.58);
\begin{scriptsize}
\fill [color=black] (-2,-2) circle (1.5pt);
\draw[color=black] (-2.1,-2.13) node {$P_0$};
\fill [color=black] (-3,1) circle (1.5pt);
\draw[color=black] (-3.2,1.09) node {$P_1$};
\fill [color=black] (2,3) circle (1.5pt);
\draw[color=black] (2.12,3.2) node {$P_2$};
\fill [color=black] (3,-2) circle (1.5pt);
\draw[color=black] (3.19,-2.05) node {$P_3$};
\fill [color=black] (-0.57,-0.21) circle (1.5pt);
\draw[color=black] (-0.7,0.04) node {$O$};
\draw[color=black] (-1.69,-0.34) node {$L_0$};
\draw[color=black] (-0.33,1.17) node {$L_1$};
\draw[color=black] (1.88,0.26) node {$L_2$};
\draw[color=black] (-0.24,-1.35) node {$L_3$};
\end{scriptsize}
\end{tikzpicture}
\caption{The centrally-projective quadrilateral $\cppol(O;P_0,P_1,P_2,P_3)$ with each one of its fields of transverse lines}
\label{fig:quadrilateral_billiard}
\end{figure}

\begin{proposition}
\label{prop:4_reflective}
Any $(M_0,M_1)\in\gamma_0\times\gamma_1$ with $M_0\neq M_1$ determines a $4$-periodic orbit of the centrally-projective quadrilateral $\cppol(O;P_0,P_1,P_2,P_3)$.
\end{proposition}

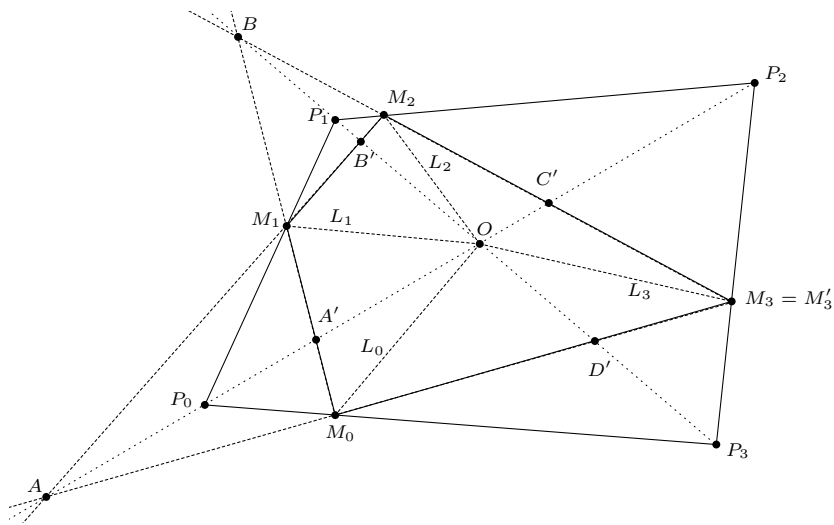
\begin{figure}[!b]
\centering
\begin{tikzpicture}[line cap=round,line join=round,>=triangle 45,x=1.7cm,y=1.7cm]
\clip(-3.8,-1.2) rectangle (3.5,2.8);
\draw (-2.27,-0.28)-- (-1.25,1.95);
\draw (-1.25,1.95)-- (2.03,2.24);
\draw (2.03,2.24)-- (1.73,-0.59);
\draw (1.73,-0.59)-- (-2.27,-0.28);
\draw [dash pattern=on 1pt off 1pt,domain=-4.36118736911263:-1.2546504345620575] plot(\x,{(--1.99--1.48*\x)/-0.38});
\draw [dash pattern=on 1pt off 1pt,domain=-4.36118736911263:-0.8682629057033406] plot(\x,{(-2.26-0.86*\x)/-0.76});
\draw [dash pattern=on 1pt off 1pt,domain=-4.36118736911263:1.850326030682974] plot(\x,{(--0.01-0.88*\x)/-3.1});
\draw [dash pattern=on 1pt off 1pt,domain=-4.36118736911263:1.850326030682974] plot(\x,{(-4.13--1.46*\x)/-2.72});
\draw [dotted,domain=-4.36118736911263:1.732822661578601] plot(\x,{(-2.64--2.55*\x)/-2.98});
\draw [dotted,domain=-4.36118736911263:2.030297933160431] plot(\x,{(-4.53-2.52*\x)/-4.3});
\draw (-0.87,1.99)-- (1.85,0.53);
\draw (1.85,0.53)-- (-1.25,-0.36);
\draw (-1.63,1.12)-- (-1.25,-0.36);
\draw (-1.63,1.12)-- (-0.87,1.99);
\draw [dash pattern=on 1pt off 1pt] (-0.12,0.98)-- (-1.63,1.12);
\draw [dash pattern=on 1pt off 1pt] (-0.12,0.98)-- (-0.87,1.99);
\draw [dash pattern=on 1pt off 1pt] (-0.12,0.98)-- (-1.25,-0.36);
\draw [dash pattern=on 1pt off 1pt] (-0.12,0.98)-- (1.85,0.53);
\begin{scriptsize}
\fill [color=black] (-2.27,-0.28) circle (1.5pt);
\draw[color=black] (-2.45,-0.24) node {$P_0$};
\fill [color=black] (-1.25,1.95) circle (1.5pt);
\draw[color=black] (-1.38,1.98) node {$P_1$};
\fill [color=black] (2.03,2.24) circle (1.5pt);
\draw[color=black] (2.2,2.31) node {$P_2$};
\fill [color=black] (1.73,-0.59) circle (1.5pt);
\draw[color=black] (1.9,-0.64) node {$P_3$};
\fill [color=black] (-0.12,0.98) circle (1.5pt);
\draw[color=black] (-0.08,1.1) node {$O$};
\fill [color=black] (-1.25,-0.36) circle (1.5pt);
\draw[color=black] (-1.21,-0.5) node {$M_0$};
\fill [color=black] (-1.63,1.12) circle (1.5pt);
\draw[color=black] (-1.79,1.17) node {$M_1$};
\fill [color=black] (-0.87,1.99) circle (1.5pt);
\draw[color=black] (-0.73,2.12) node {$M_2$};
\fill [color=black] (-3.51,-1) circle (1.5pt);
\draw[color=black] (-3.6,-0.91) node {$A$};
\fill [color=black] (-2.01,2.6) circle (1.5pt);
\draw[color=black] (-1.92,2.7) node {$B$};
\fill [color=black] (1.85,0.53) circle (1.5pt);
\draw[color=black] (2.3,0.55) node {$M_3=M_3'$};
\fill [color=black] (0.42,1.3) circle (1.5pt);
\draw[color=black] (0.41,1.5) node {$C'$};
\fill [color=black] (0.78,0.22) circle (1.5pt);
\draw[color=black] (0.82,0.01) node {$D'$};
\fill [color=black] (-1.4,0.23) circle (1.5pt);
\draw[color=black] (-1.3,0.43) node {$A'$};
\fill [color=black] (-1.05,1.78) circle (1.5pt);
\draw[color=black] (-1.02,1.64) node {$B'$};
\draw[color=black] (-1.2,1.19) node {$L_1$};
\draw[color=black] (-0.43,1.61) node {$L_2$};
\draw[color=black] (-0.96,0.17) node {$L_0$};
\draw[color=black] (1.13,0.61) node {$L_3$};
\end{scriptsize}
\end{tikzpicture}
\caption{The centrally-projective quadrilateral $\cppol(O;P_0,P_1,P_2,P_3)$ with a periodic orbit obtained by reflecting $M_0M_1$ three times. Here the notations are the same as in the proof of Proposition \ref{prop:4_reflective}.}
\label{fig:quadrilateral_proof}
\end{figure}

\begin{proof}
Let $M_2\in\gamma_2$ such that $M_0M_1$ is reflected into $M_1M_2$ by the reflection law at $M_1$. Let $M_3\in\gamma_3$ such that $M_1M_2$ is reflected into $M_2M_3$ by the reflection law at $M_2$. Let $M_3'\in\gamma_3$ such that $M_0M_1$ is reflected into $M_0M_3'$ by the reflection law at $M_0$. Denote by $d$ the line reflected from $M_2M_3$ by the projective reflection law at $M_3$. We have to show that $d=M_0M_3'$. 

First, let us introduce a few notations (see Figure \ref{fig:quadrilateral_proof}). Consider the line $M_0M_1$; it intersects: the line $P_0P_2$ at a point $B$ and the line $P_1P_3$ at a point $A'$. Now consider the line $M_1M_2$; it intersects: the line $P_0P_2$ at a point $B'$ and the line $P_1P_3$ at a point $A$. Finally let $C'$ be the intersection point of $M_2M_3$ with $P_1P_3$ and $D'$ the intersection point of $M_0M_3'$ with $P_0P_2$.

Then, notice that by the projective law of reflection at $M_1$, the quadruple of points $(A, A',P_1,O)$ is harmonic. Since the points $P_1,A',O$ correpond to the lines $\gamma_0$, $M_0M_1$, $L_0(M_0)$, the previously defined reflected line $M_0M_3'$ needs to pass through $A$ in order to form a harmonic quadruple of lines. The same remark on the other diagonal leads to note that $M_2M_3$ passes through $B$.

Now by the reflection law at $M_2$, one observe that the quadruple of points $(A,C',O,P_3)$ is harmonic. But $\gamma_3$ passes through $P_3$, $M_3M_2$ through $C'$ and $L_3(M_3)$ through $O$. Hence $d$ needs to pass through $A$. Then, by the reflection law at $M_0$, one observe that the quadruple of points $(B,D',O,P_0)$ is harmonic. But $\gamma_3$ passes through $P_0$, $M_3M_2$ through $B$ and $L_3(M_3)$ through $O$. Hence $d$ needs to pass through $D'$.

Therefore we conclude that $d=AD'=M_0M_3'$.
\end{proof}

\begin{remark}
Notice that the spherical billiard cannot be deduced from a usual billiard on the sphere by the same construction as described for the right-spherical billiard.
\end{remark}

\section{centrally-projective regular polygons}
\label{sec:n_even}

In the following, we prove Proposition \ref{res:regular_centrally_reflective}. We first introduce some notations:

Let $n=2k$, $k\geq 2$, an even integer. Let $\mathcal{P}$ be an $n$-sided regular polygon of radius $1$, whose vertices are clockwise denoted by $P_0, P_1,\ldots,P_{n-1}$. In the following, each time we refer to the index $i$, we will consider it modulo $n$, that is we identify $i$ and its rest when divided by $n$. Define $O$ to be the intersection of the great diagonals of the polygon, which are the lines $P_iP_{i+k}$, $i=0,\ldots,k-1$. Construct a field of transverse lines $L_i$ on $\gamma_i:=P_iP_{i+1}$ by setting $L_i(M)=MO$ (the line joining the basepoint $M$ and $O$). The projective billiard table defined by the $\gamma_i$ and the $L_i$ is what we call the \textit{centrally-projective regular polygon} $\cppol(O;P_0,\ldots,P_{n-1})$, see Figure \ref{fig:polygon_presentation}.

Let $(M_0,M_1)\in P_0P_1\times P_1P_2$ with $M_0\neq M_1$. We want to show that the orbit $(M_m)_{m\in\integer}$ of $(M_0,M_1)$ is $n$-periodic. We first prove the

\begin{figure}[!h]
\centering
\begin{tikzpicture}[line cap=round,line join=round,>=triangle 45,x=2.5cm,y=2.5cm]
\clip(-2.27,-1.21) rectangle (1.99,1.16);
\draw (-0.89,0.5)-- (0,1);
\draw (0.89,0.5)-- (0,1);
\draw (-0.89,-0.5)-- (0,-1);
\draw (0.89,-0.5)-- (0,-1);
\draw (-0.89,0.5)-- (-0.89,-0.5);
\draw (0.89,0.5)-- (0.89,-0.5);
\draw [dotted] (0,1)-- (0,-1);
\draw [dotted] (-0.89,-0.5)-- (0.89,0.5);
\draw [dotted] (-0.89,0.5)-- (0.89,-0.5);
\draw (-0.89,0.09)-- (-0.44,0.75);
\draw (-0.44,0.75)-- (0.28,0.84);
\draw (0.28,0.84)-- (0.89,-0.38);
\draw (0.89,-0.38)-- (0.75,-0.57);
\draw (0.75,-0.57)-- (-0.48,-0.73);
\draw [dash pattern=on 1pt off 1pt] (0,0)-- (-0.44,0.75);
\draw [dash pattern=on 1pt off 1pt] (0,0)-- (0.28,0.84);
\draw [dash pattern=on 1pt off 1pt] (0,0)-- (0.89,-0.38);
\draw [dash pattern=on 1pt off 1pt] (0,0)-- (0.75,-0.57);
\draw [dash pattern=on 1pt off 1pt] (0,0)-- (-0.48,-0.73);
\draw [dash pattern=on 1pt off 1pt] (0,0)-- (-0.89,0.09);
\begin{scriptsize}
\fill [color=black] (-0.89,0.5) circle (0.5pt);
\draw[color=black] (-0.94,0.55) node {$P_1$};
\fill [color=black] (0,1) circle (0.5pt);
\draw[color=black] (-0.01,1.06) node {$P_2$};
\fill [color=black] (0.89,0.5) circle (0.5pt);
\draw[color=black] (0.95,0.54) node {$P_3$};
\fill [color=black] (0,1) circle (0.5pt);
\fill [color=black] (-0.89,-0.5) circle (0.5pt);
\draw[color=black] (-0.93,-0.52) node {$P_0$};
\fill [color=black] (0,-1) circle (0.5pt);
\draw[color=black] (0.01,-1.05) node {$P_5$};
\fill [color=black] (0.89,-0.5) circle (0.5pt);
\draw[color=black] (0.95,-0.53) node {$P_4$};
\fill [color=black] (0,-1) circle (0.5pt);
\fill [color=black] (0,0) circle (0.5pt);
\draw[color=black] (0.05,0.06) node {$O$};
\fill [color=black] (-0.89,0.09) circle (0.5pt);
\draw[color=black] (-0.96,0.1) node {$M_0$};
\fill [color=black] (-0.44,0.75) circle (0.5pt);
\draw[color=black] (-0.47,0.81) node {$M_1$};
\fill [color=black] (0.28,0.84) circle (0.5pt);
\draw[color=black] (0.31,0.89) node {$M_2$};
\fill [color=black] (0.89,-0.38) circle (0.5pt);
\draw[color=black] (0.95,-0.36) node {$M_3$};
\fill [color=black] (0.75,-0.57) circle (0.5pt);
\draw[color=black] (0.8,-0.62) node {$M_4$};
\fill [color=black] (-0.48,-0.73) circle (0.5pt);
\draw[color=black] (-0.5,-0.78) node {$M_5$};
\draw[color=black] (-0.28,0.4) node {$L_1$};
\draw[color=black] (0.22,0.42) node {$L_2$};
\draw[color=black] (0.5,-0.12) node {$L_3$};
\draw[color=black] (0.37,-0.32) node {$L_4$};
\draw[color=black] (-0.29,-0.35) node {$L_5$};
\draw[color=black] (-0.42,0.11) node {$L_0$};
\end{scriptsize}
\end{tikzpicture}
\caption{A centrally-projective polygon $\cppol(O;P_0,\ldots,P_{5})$ and a piece of trajectory after four projective bounces.}
\label{fig:polygon_presentation}
\end{figure}
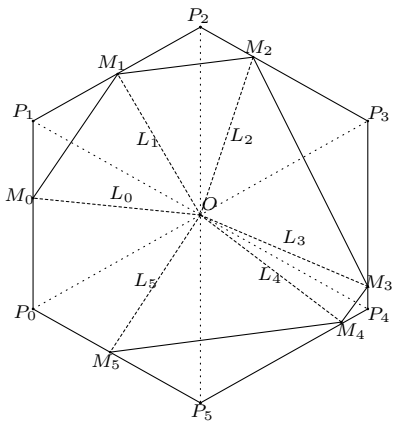

\begin{lemma}
\label{lemma:dtes_concourantes}
Fix an $m\in\pinteger$ and consider the great diagonal $L_m =P_mP_{m+k}$. Then for any $r\in\pinteger$, the lines $M_{m-r-2}M_{m-r-1}$ and $M_{m+r}M_{m+r+1}$ intersect $L_m$ at a same point.
\end{lemma}

\begin{figure}[!h]
\centering
\begin{tikzpicture}[line cap=round,line join=round,>=triangle 45,x=2.5cm,y=2.5cm]
\clip(-2.4,-0.65) rectangle (2.06,1.6);
\draw [line width=1.2pt] (0,0.8)-- (-0.8,0.6);
\draw [line width=1.2pt] (-0.8,0.6)-- (-1.4,0.2);
\draw [line width=1.2pt] (-1.4,0.2)-- (-1.8,-0.4);
\draw [line width=1.2pt] (0,0.8)-- (0.8,0.6);
\draw [line width=1.2pt] (0.8,0.6)-- (1.4,0.2);
\draw [line width=1.2pt] (1.4,0.2)-- (1.8,-0.4);
\draw [dash pattern=on 1pt off 1pt,domain=-1.0812456489024758:2.055306508092342] plot(\x,{(--0.62--0.3*\x)/0.72});
\draw [dash pattern=on 1pt off 1pt,domain=-1.5015993354037684:2.055306508092342] plot(\x,{(--0.57--0.36*\x)/0.42});
\draw [dash pattern=on 1pt off 1pt,domain=-2.0013299841538337:1.0812456489024758] plot(\x,{(-0.62--0.3*\x)/-0.72});
\draw [dash pattern=on 1pt off 1pt,domain=-2.0013299841538337:1.5015993354037684] plot(\x,{(-0.57--0.36*\x)/-0.42});
\draw [dotted] (0,-0.6) -- (0,1.6);
\draw (-1.5,0.05)-- (-1.08,0.41);
\draw (-1.08,0.41)-- (-0.36,0.71);
\draw (0.36,0.71)-- (1.08,0.41);
\draw (1.08,0.41)-- (1.5,0.05);
\draw (-0.36,0.71)-- (0.36,0.71);
\begin{scriptsize}
\fill [color=black] (0,0.8) circle (1.5pt);
\draw[color=black] (-0.03,0.6) node {$P_m$};
\fill [color=black] (-0.8,0.6) circle (1.5pt);
\draw[color=black] (-0.86,0.68) node {$P_{m-r}$};
\fill [color=black] (-1.4,0.2) circle (1.5pt);
\draw[color=black] (-1.6,0.27) node {$P_{m-r-1}$};
\fill [color=black] (-1.8,-0.4) circle (1.5pt);
\draw[color=black] (-2.07,-0.37) node {$P_{m-r-2}$};
\fill [color=black] (0.8,0.6) circle (1.5pt);
\draw[color=black] (0.83,0.67) node {$P_{m+r}$};
\fill [color=black] (1.4,0.2) circle (1.5pt);
\draw[color=black] (1.6,0.25) node {$P_{m+r+1}$};
\fill [color=black] (1.8,-0.4) circle (1.5pt);
\draw[color=black] (1.87,-0.37) node {$P_{m+r+2}$};
\fill [color=black] (-1.08,0.41) circle (1.5pt);
\draw[color=black] (-1.3,0.48) node {$M_{m-r-1}$};
\fill [color=black] (-0.36,0.71) circle (1.5pt);
\draw[color=black] (-0.37,0.83) node {$M_{m-r}$};
\fill [color=black] (-1.5,0.05) circle (1.5pt);
\draw[color=black] (-1.8,0.07) node {$M_{m-r-2}$};
\fill [color=black] (1.08,0.41) circle (1.5pt);
\fill [color=black] (0.36,0.71) circle (1.5pt);
\draw[color=black] (0.39,0.83) node {$M_{m+r-1}$};
\fill [color=black] (1.5,0.05) circle (1.5pt);
\draw[color=black] (1.8,0.09) node {$M_{m+r+1}$};
\fill [color=black] (1.08,0.41) circle (1.5pt);
\draw[color=black] (1.3,0.46) node {$M_{m+r}$};
\fill [color=black] (0,0) circle (0.5pt);
\draw[color=black] (0.15,0) node {$L_m$};
\end{scriptsize}
\end{tikzpicture}
\caption{As in the proof of Proposition \ref{prop:even_polygon_reflective}, since the lines $M_{m-r-1}M_{m-r}$ and $M_{m+r-1}M_{m+r}$ intersect $L_m$ at the same point, the lines  $M_{m-r-2}M_{m-r-1}$ and $M_{m+r}M_{m+r+1}$ also intersect $L_m$ at a same point.}
\label{fig:polygon_proof}
\end{figure}
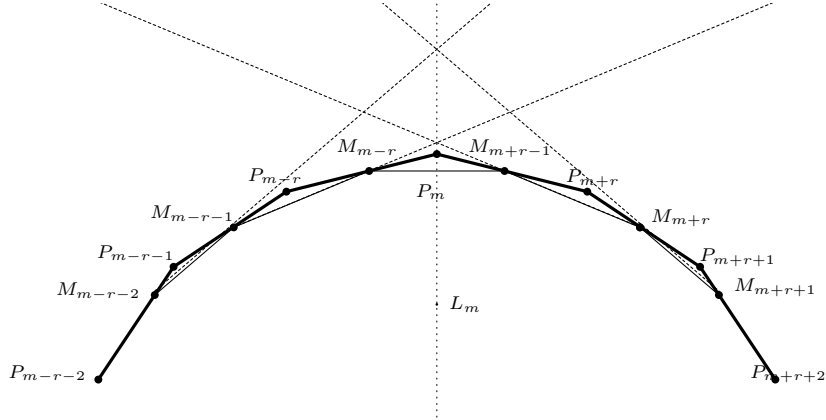

\begin{proof}
Let us prove Lemma \ref{lemma:dtes_concourantes} by induction on $r$.

\textsl{Case when $r=0$:} Fix any $m\in\integer$. Let $A$ be the intersection point of $M_{m-2}M_{m-1}$ with $L_m$, $A'$ of $M_mM_{m+1}$ with $L_m$ and $B$ of $M_{m-1}M_m$. Consider harmonic quadruples of points on $L_m$: $(A,B,P_m,O)$ is harmonic by the reflection law in $M_{m-1}$, and $(A',B,P_m,O)$ is harmonic by the reflection law in $M_{m+1}$. Hence $A=A'$ which concludes the proof for $r=0$.

\textsl{Inductive step:} suppose Lemma \ref{lemma:dtes_concourantes} is true for any $r'<r$ and let us prove it for $r$. See Figure \ref{fig:polygon_proof} for a detailled drawing of the situation. Fix an $m\in\integer$. By assumption, we know that $M_{m-r-1}M_{m-r}$ and $M_{m+r-1}M_{m+r}$ intersect $L_m$ at a same point $A$. Furthermore by symmetry of the regular polygon through $L_m$, the lines $P_{m-r-1}P_{m-r}$ and $P_{m+r}P_{m+r+1}$ intersect $L_m$ at the same point. Now, the three following lines through $M_{m-r-1}$, $M_{m-r-1}M_{m-r}$, $P_{m-r-1}P_{m-r}$, $M_{m-r-1}O$ intersect $L_m$ through the same points than the three following lines through $M_{m+r}$, $M_{m+r-1}M_{m+r}$, $P_{m+r}P_{m+r+1}$, $M_{m+r}O$. Hence in order to sastisfy the projection law at $M_{m-r-1}$ and $M_{m+r}$ respectively, the lines $M_{m-r-1}M_{m-r-2}$ and $M_{m+r}M_{m+r+1}$ should intersect at the same point. Hence the inductive step is over and this conclude the proof.
\end{proof}

\begin{proposition}
\label{prop:even_polygon_reflective}
Any $(M_0,M_1)\in P_0P_1\times P_1P_2$ with $M_0\neq M_1$ determines an $n$-periodic orbit of the centrally-projective regular polygon $\cppol(O;P_0,\ldots,P_{n-1})$.
\end{proposition}

\begin{proof}
We have to show that $M_{-1}M_0 = M_{n-1}M_{n}$. We will use Lemma \ref{lemma:dtes_concourantes}. First, by setting $m=k$ and $r=k-1$, we conclude that the lines 
$M_0M_{-1}$ and the lines $M_{n-1}M_{n}$ intersect $L_k$ at a same point denoted by $A$. Then, by setting $m=k+1$ and $r=k-2$ we get that the lines $M_1M_2$ and $M_{n-1}M_n$ intersect $L_{k+1}=L_1$ at a same point $B$. Now it is also true that $M_0M_{-1}$ intersect $L_1$ at $B$, by Lemma \ref{lemma:dtes_concourantes} and setting $m=1$ and $r=0$. Hence we have shown that $M_{n-1}M_n=AB=M_0M_{-1}$ which concludes the proof.
\end{proof}

\begin{remark}
We note that the assumptions that $\mathcal{P}$ is regular and that $O$ is on the intersection of the great diagonals is fundamental to guarantee the reflectivity. On a simulation, see for example \cite{fierobe_simulation_projective_billiards}, the orbits are destroyed after small perturbations of $\mathcal{P}$'s vertices or of the origin. However, the next example we describe is much more stable in this sense.
\end{remark}

\section{Centrally-projective polygons and dual billiards}
\label{sec:centrally_proj_poly}

In this section we prove Proposition \ref{res:centrally_2n_reflective}. We first introduce some notations:

Let $n=2k+1$, $k\geq 1$, be an odd integer. Let $\mathcal{P}$ be an $n$-sided polygon, whose vertices are clockwise denoted by $P_0, P_1,\ldots,P_{n-1}$. We suppose that any three consececutive vertices are not on the same line. Choose another point $O$, which do not lie on any line of the type $\gamma_i:=P_iP_{i+1}$ ($i$ is taken modulo $n$). Construct a field of transverse lines $L_i$ on $\gamma_i$ by setting $L_i(M)=MO$ (the line joining the basepoint $M$ and $O$). The projective billiard table defined by the $\gamma_i$ and the $L_i$ is what we call the \textit{centrally-projective polygon} $\cppol(O;P_0,\ldots,P_{n-1})$ (see Figure \ref{fig:centrally_proj_general}).

We will prove the

\begin{proposition}
\label{prop:centrally_polygon_reflective}
Let $n=2k+1\geq 3$. Any $(M_0,M_1)\in P_0P_1\times P_1P_2$ with $M_0\neq M_1$ determines a $2n$-periodic orbit of the centrally-projective polygon $\cppol(O;P_0,\ldots,P_{n-1})$.
\end{proposition}

\begin{figure}[!h]
\centering
\begin{tikzpicture}[line cap=round,line join=round,>=triangle 45,x=3.5cm,y=3.5cm]
\clip(-1.2,-0.64) rectangle (1.2,1.09);
\draw (0,1)-- (-0.94,-0.57);
\draw (-0.94,-0.57)-- (1,0.2);
\draw (1,0.2)-- (0,1);
\draw [dotted,domain=-1.9:1.7] plot(\x,{(-0.15--0.72*\x)/-0.15});
\draw [dotted,domain=-1.9:1.7] plot(\x,{(--0.25-0.08*\x)/0.85});
\draw [dotted,domain=-1.9:1.7] plot(\x,{(-0.18-0.85*\x)/-1.09});
\draw (0.07,-0.17)-- (-0.23,0.61);
\draw (-0.23,0.61)-- (0.75,0.4);
\draw (0.75,0.4)-- (0.68,0.07);
\draw (0.68,0.07)-- (-0.56,0.06);
\draw (0.47,0.62)-- (0.07,-0.17);
\draw (-0.56,0.06)-- (0.47,0.62);
\begin{scriptsize}
\fill [color=black] (0,1) circle (1.5pt);
\draw[color=black] (0.11,1.03) node {$P_0$};
\fill [color=black] (-0.94,-0.57) circle (1.5pt);
\draw[color=black] (-1.03,-0.56) node {$P_2$};
\fill [color=black] (1,0.2) circle (1.5pt);
\draw[color=black] (1.07,0.25) node {$P_1$};
\fill [color=black] (0.07,-0.17) circle (1.5pt);
\draw[color=black] (0.1,-0.25) node {$M_4$};
\fill [color=black] (-0.23,0.61) circle (1.5pt);
\draw[color=black] (-0.13,0.66) node {$M_5$};
\fill [color=black] (0.75,0.4) circle (1.5pt);
\draw[color=black] (0.83,0.46) node {$M_0$};
\fill [color=black] (0.15,0.28) circle (1.5pt);
\draw[color=black] (0.19,0.36) node {$O$};
\fill [color=black] (0.68,0.07) circle (1.5pt);
\draw[color=black] (0.75,0.01) node {$M_1$};
\fill [color=black] (-0.56,0.06) circle (1.5pt);
\draw[color=black] (-0.64,0.08) node {$M_2$};
\fill [color=black] (0.47,0.62) circle (1.5pt);
\draw[color=black] (0.51,0.69) node {$M_3$};
\end{scriptsize}
\end{tikzpicture}
\caption{A $6$-periodic orbit $(M_k)_k$ on a centrally-projective triangle $P_0P_1P_2$ with origin $O$. The dotted lines are representatives of the transverse fields of lines on the sides of the triangle.}
\label{fig:centrally_proj_general}
\end{figure}
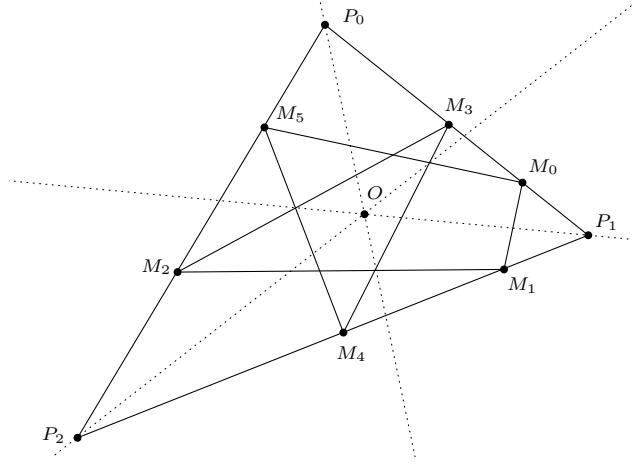

\subsection*{Outer ghost billiard associated to a centrally pojective billiard}

We can associate to any projective billiard a billiard called \textit{dual}, or \textit{outer billiard}, and this construction is well-explained in \cite{taba_projectif}. We will explain from the beginning the constructions which we will use, but for now, the already enlightened reader should keep in mind that the version of outer billiard we will use is very simplified, and the latter should be better called \textit{outer ghost billiard}. Usual outer billiards are much more complicated, see \cite{taba_dual1, taba_dual2} for more details. Let us first recall some notions about polar duality.

Fix a point $O\in\real^2$ and complete $\real^2$ with a line at infinity $\lambda$ to form $\rp^2$. Consider the polar duality with respect to the point $O$. It sends any point $p\in\rp^2$ to a line $p\etoile$ and any line $\ell\subset\rp^2$ to a point $\ell\etoile$ in a bijective way, such that $O\etoile=\lambda$ and $\lambda\etoile=O$. It is involutive in the sense that if you apply the polar duality with respect to $O$ two times, you get the identity: $p^{\ast\ast}=p$ and $\ell^{\ast\ast}=\ell$. Finaly polar duality has the incidence property: $p\in\ell$ if and only if $\ell\etoile\in p\etoile$.

\vsp
\noindent Now as in Figure \ref{fig:line_inter}, suppose that we are given four lines of $\real^2$, $\ell$, $\ell'$, $T$ and $L$ such that
\begin{itemize}
\item[--] $L$ passes through $O$;
\item[--] The four lines $\ell,\ell',L,T$ pass through the same point $p\neq O$ and are harmonic.
\end{itemize}
By the incidence property, the points $\ell\etoile$, $\ell^{'\ast}$, $L\etoile$ and $T\etoile$ belong to the line $p\etoile$, and are harmonic since the corresponding four lines are. But also $L\etoile\in O\etoile=\lambda$ is at infinity. By harmonicity, both vectors $\overrightarrow{T\etoile\ell\etoile}$ and $\overrightarrow{T\etoile\ell^{'\ast}}$ are opposite.

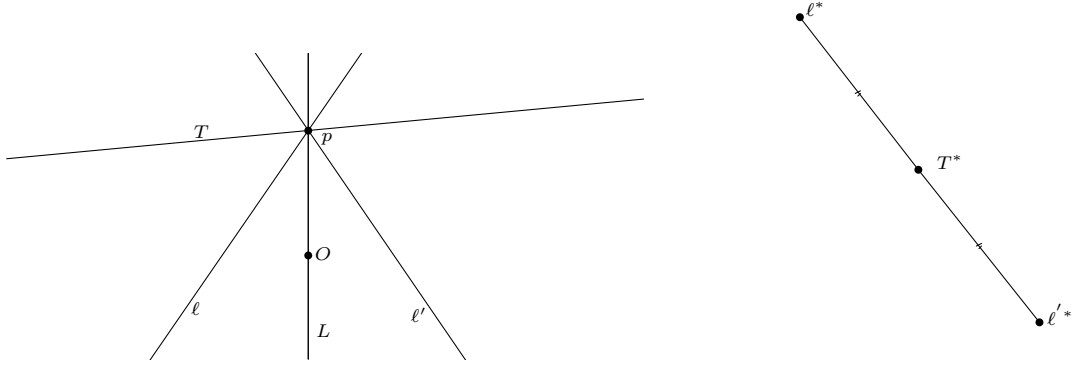
\begin{figure}[!h]
\centering
\begin{tikzpicture}[line cap=round,line join=round,>=triangle 45,x=1.0cm,y=1.0cm]
\clip(-4.01,-1.38) rectangle (4.46,2.69);
\draw [domain=-4.01:4.46] plot(\x,{(--5.32--0.3*\x)/3.2});
\draw [domain=-4.01:4.46] plot(\x,{(-2.76-2.41*\x)/-1.66});
\draw [domain=-4.01:4.46] plot(\x,{(--2.78-2.44*\x)/1.67});
\draw (0,-1.38) -- (0,2.69);
\draw (0,-1.38) -- (0,2.69);
\begin{scriptsize}
\fill [color=black] (0,0) circle (1.5pt);
\draw[color=black] (0.2,0.02) node {$O$};
\draw[color=black] (0.2,-1) node {$L$};
\draw[color=black] (-1.42,1.65) node {$T$};
\fill [color=black] (0,1.66) circle (1.5pt);
\draw[color=black] (0.25,1.55) node {$p$};
\draw[color=black] (-1.50,-0.71) node {$\ell$};
\draw[color=black] (1.47,-0.79) node {$\ell'$};
\end{scriptsize}
\end{tikzpicture}
\hspace{1cm}
\begin{tikzpicture}[line cap=round,line join=round,>=triangle 45,x=3.5cm,y=3.5cm]
\clip(-1.4,0) rectangle (0,1.4);
\draw (-1.16,1.3)-- (-0.71,0.72);
\draw (-0.93,1.02) -- (-0.95,1.01);
\draw (-0.93,1.01) -- (-0.94,1);
\draw (-0.71,0.72)-- (-0.25,0.14);
\draw (-0.47,0.44) -- (-0.49,0.43);
\draw (-0.47,0.43) -- (-0.48,0.42);
\begin{scriptsize}
\fill [color=black] (-1.16,1.3) circle (1.5pt);
\draw[color=black] (-1.1,1.33) node {$\ell\etoile$};
\fill [color=black] (-0.25,0.14) circle (1.5pt);
\draw[color=black] (-0.17,0.17) node {$\ell^{'\ast}$};
\fill [color=black] (-0.71,0.72) circle (1.5pt);
\draw[color=black] (-0.59,0.75) node {$T\etoile$};
\end{scriptsize}
\end{tikzpicture}
\caption{On the left: four harmonic lines $\ell$, $\ell'$, $L$ and $T$ through $p$. On the right: their polar duals with respect to $O$, $T\etoile$ is in the middle of the segment joining $\ell\etoile$ and $\ell^{'\ast}$, $L\etoile$ is at infinity.}
\label{fig:line_inter}
\end{figure}

We can apply this to a centrally-projective polygon $\cppol(O;P_0,\ldots,P_{n-1})$: denote by $Q_0,\ldots,Q_{n-1}$ the dual points of its sides, where $Q_k=P_kP_{k+1}\etoile$ ($k$ is tajen modulo $n$). Choose a projective orbit $(M_k)_{k\in\integer}$ of $\cppol(O;P_0,\ldots,P_{n-1})$, an consider, for all $k$, the dual point $N_k$ of the line $M_{k-1}M_{k}$; we call $(N_k)_k$ the dual orbit of $(M_k)_k$. Now we are ready to define the \textit{outer ghost billiard} on $Q_0\cdots Q_{n-1}$ (see Figure \ref{fig:outer_orbit}):

\begin{definition}
To any point $N_0$ distinct from $Q_0,\ldots, Q_{n-1}$, we associate the sequence of points $(N_k)_k$, called \textit{ghost outer orbit of the polygon $Q_0\cdots Q_{n-1}$}, and uniquely determined by the relation
$$\overrightarrow{N_kQ_k}=\overrightarrow{Q_kN_{k+1}}$$
for all $k\in\integer$ (where here $Q_k$ is used for $Q_{k\text{ mod }n}$). The orbit is said to be \textit{$m$-periodic} if $N_{m}=N_0$.
\end{definition}

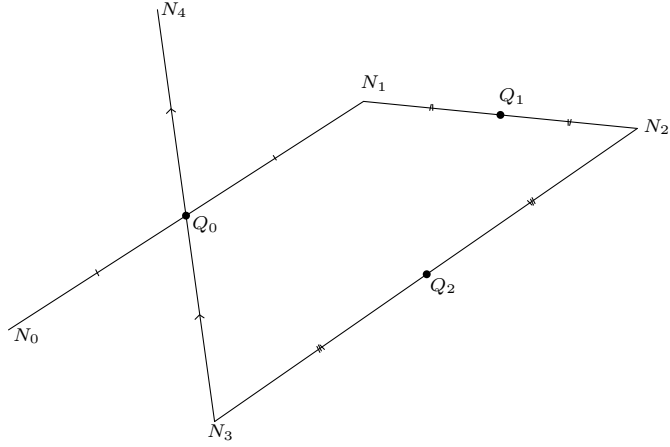
\begin{figure}[!h]
\centering
\begin{tikzpicture}[line cap=round,line join=round,>=triangle 45,x=2.0cm,y=2.0cm]
\clip(-2.4,-0.7) rectangle (2.4,2.4);
\draw (-2.27,0.07)-- (-1.09,0.83);
\draw (-1.69,0.47) -- (-1.67,0.43);
\draw (0.09,1.59)-- (-1.09,0.83);
\draw (-0.49,1.2) -- (-0.51,1.23);
\draw (0.09,1.59)-- (1,1.5);
\draw (0.54,1.57) -- (0.53,1.53);
\draw (0.55,1.57) -- (0.55,1.53);
\draw (1,1.5)-- (1.91,1.41);
\draw (1.45,1.47) -- (1.45,1.43);
\draw (1.47,1.47) -- (1.46,1.43);
\draw (1.91,1.41)-- (0.51,0.44);
\draw (1.23,0.91) -- (1.21,0.95);
\draw (1.22,0.9) -- (1.2,0.94);
\draw (1.21,0.9) -- (1.18,0.93);
\draw (0.51,0.44)-- (-0.9,-0.54);
\draw (-0.17,-0.06) -- (-0.2,-0.03);
\draw (-0.19,-0.07) -- (-0.21,-0.03);
\draw (-0.2,-0.08) -- (-0.22,-0.04);
\draw (-0.9,-0.54)-- (-1.09,0.83);
\draw (-1,0.17) -- (-0.97,0.15);
\draw (-1,0.17) -- (-1.03,0.14);
\draw (-1.09,0.83)-- (-1.28,2.2);
\draw (-1.19,1.54) -- (-1.16,1.52);
\draw (-1.19,1.54) -- (-1.22,1.51);
\begin{scriptsize}
\fill [color=black] (-1.09,0.83) circle (1.5pt);
\draw[color=black] (-0.96,0.77) node {$Q_0$};
\fill [color=black] (1,1.5) circle (1.5pt);
\draw[color=black] (1.08,1.62) node {$Q_1$};
\fill [color=black] (0.51,0.44) circle (1.5pt);
\draw[color=black] (0.62,0.36) node {$Q_2$};
\draw[color=black] (-2.15,0.03) node {$N_0$};
\draw[color=black] (0.16,1.71) node {$N_1$};
\draw[color=black] (2.04,1.42) node {$N_2$};
\draw[color=black] (-0.86,-0.61) node {$N_3$};
\draw[color=black] (-1.17,2.2) node {$N_4$};
\end{scriptsize}
\end{tikzpicture}
\caption{Five successive points, $N_k$ with $k=0\ldots 4$, of an outer ghost orbit of $Q_0Q_1Q_2$.}
\label{fig:outer_orbit}
\end{figure}

By previous discussion, we have the

\begin{proposition}
\label{prop:dual_association}
If $(M_k)_k$ is an orbit of $\cppol(O;P_0,\ldots,P_{n-1})$, its dual orbit $(N_k)_k$ is an outer ghost orbit of $Q_0\cdots Q_{n-1}$. Furthermore, $(M_k)_k$ is $m$-periodic if and only if $(N_k)_k$ is $m$-periodic.
\end{proposition}

Therefore, to prove Proposition \ref{prop:centrally_polygon_reflective}, we only need to show the

\begin{proposition}
\label{prop:reflect_outer}
Let $n=2k+1\geq 3$. Any outer ghost orbit of $Q_0\cdots Q_{n-1}$ is $2n$-periodic.
\end{proposition}

\begin{proof}
Let $(N_k)_k$ be an outer ghost orbit of $Q_0\cdots Q_{n-1}$. We first show the

\begin{lemma}
\label{lemma:thales}
For $k\in\integer$, we have $\overrightarrow{N_kN_{k+2}}=2\overrightarrow{Q_kQ_{k+1}}$.
\end{lemma}

\begin{proof}
Here we are in a Thales configuration and the proof follows. Indeed, by Chasles relation and the equalities $\overrightarrow{N_kQ_k}=\overrightarrow{Q_kN_{k+1}}$, $\overrightarrow{N_{k+1}Q_{k+1}}=\overrightarrow{Q_{k+1}N_{k+2}}$, we can write that
$$\overrightarrow{N_kN_{k+2}}=\overrightarrow{N_kQ_k}+\overrightarrow{Q_kN_{k+1}}+\overrightarrow{N_{k+1}Q_{k+1}}+\overrightarrow{Q_{k+1}N_{k+2}} = 2\overrightarrow{Q_kN_{k+1}}+2\overrightarrow{N_{k+1}Q_{k+1}}=2\overrightarrow{Q_kQ_{k+1}}.$$
\end{proof}

\noindent To conclude the proof, we need to show that $N_0=N_{2n}$. Indeed by Lemma \ref{lemma:thales} we have
$$\overrightarrow{N_0N_{2n}} = \sum_{j=0}^{n-1}\overrightarrow{N_{2j}N_{2j+2}} = 2\sum_{j=0}^{n-1}\overrightarrow{Q_{2j}Q_{2j+1}}$$
and we want to show that this is equal to $0$. By splitting the latter sum at $k$, where $n-1=2k$, we get
$$\sum_{j=0}^{n-1}\overrightarrow{Q_{2j}Q_{2j+1}} = \sum_{j=0}^{k}\overrightarrow{Q_{2j}Q_{2j+1}} + \sum_{j=k+1}^{n-1}\overrightarrow{Q_{2j}Q_{2j+1}}$$
and the latter sum can be rewritten with the change of index $j'=j-k$  since $\overrightarrow{Q_{2j}Q_{2j+1}}=\overrightarrow{Q_{2j-n}Q_{2j+1-n}}=\overrightarrow{Q_{2j'-1}Q_{2j'}}$, to get 
$$\sum_{j=0}^{n-1}\overrightarrow{Q_{2j}Q_{2j+1}} = \sum_{j=0}^{k}\overrightarrow{Q_{2j}Q_{2j+1}} + \sum_{j'=1}^{k+1}\overrightarrow{Q_{2j'-1}Q_{2j'}}=0.$$
which concludes the proof.
\end{proof}

\begin{proof}[Proof of Proposition \ref{prop:centrally_polygon_reflective}]
Since any outer ghost orbit of $Q_0\cdots Q_{n-1}$ is $2n$-periodic (Proposition \ref{prop:reflect_outer}), so is any projective orbit of $\cppol(O;P_0,\ldots,P_{n-1})$ by Proposition \ref{prop:dual_association}.
\end{proof}

\section*{Acknowledgments}

I'm very grateful to Simon Allais for his explanations about harmonic lines and how to prove Proposition \ref{prop:3_reflective} with this method, and Alexey Glutsyuk for its very helpful advices.

\end{document}